\documentclass[12pt]{amsart}
\usepackage[latin1]{inputenc}
\usepackage{indentfirst}
\usepackage[pagewise,mathlines]{lineno}
\usepackage{amsfonts}
\usepackage{graphicx}
\usepackage{amsmath}
\usepackage{amssymb}
\usepackage{latexsym}
\usepackage{amscd}
\usepackage{xcolor}
\usepackage{float}
\usepackage[pagebackref]{hyperref}
\usepackage[linesnumbered,lined,ruled,commentsnumbered]{algorithm2e}

\newcommand{\F}{\mathbb {F}}

\newtheorem{theorem}{Theorem}[section]
\newtheorem{definition}[theorem]{Definition}
\newtheorem{lemma}[theorem]{Lemma}

\newtheorem{proposition}[theorem]{Proposition}
\newtheorem{remark}[theorem]{Remark}

\begin{document}


\title[$r$-Primitive $k$-Normal elements in Arithmetic Progressions]{$r$-Primitive $k$-Normal elements in Arithmetic Progressions over Finite Fields}

\author{Josimar J.R. Aguirre$^1$, Ab\'ilio Lemos$^2$, Victor G.L. Neumann$^1$ and S\'avio Ribas$^3$}

\maketitle

\vspace{1ex}

\small{$^1$Faculdade de Matem\'{a}tica, Universidade Federal de Uberl\^{a}ndia, 38.408-902 Uberl\^{a}ndia-MG, Brazil, \url{josimar.mat@ufu.br}, \url{victor.neumann@ufu.br}}

\small{$^2$Departamento de Matem\'{a}tica, Universidade Federal de Vi\c{c}osa, 36570-900 Vi\c{c}osa-MG, Brazil, \url{abiliolemos@ufv.br}}

\small{$^3$Permanent: Departamento de Matem\'{a}tica, Universidade Federal de Ouro Preto, 35.400-000 Ouro Preto-MG, Brazil, \url{savio.ribas@ufop.edu.br}}

\small{$^3$Current: Institute of Mathematics and Scientific Computing, University of Graz, 8010 Graz, Styria, Austria.}

\begin{abstract}
Let $\mathbb{F}_{q^n}$ be a finite field with $q^n$ elements. For a positive divisor $r$ of $q^n-1$, the element $\alpha \in \mathbb{F}_{q^n}^*$ is called \textit{$r$-primitive} if its multiplicative order is $(q^n-1)/r$. Also, for a non-negative integer $k$, the element $\alpha \in \mathbb{F}_{q^n}$ is \textit{$k$-normal} over $\mathbb{F}_q$ if $\gcd(\alpha x^{n-1}+ \alpha^q x^{n-2} +  \ldots + \alpha^{q^{n-2}}x + \alpha^{q^{n-1}} , x^n-1)$ in $\mathbb{F}_{q^n}[x]$ has degree $k$. In this paper we discuss the existence of elements in arithmetic progressions $\{\alpha, \alpha+\beta, \alpha+2\beta, \ldots\alpha+(m-1)\beta\} \subset \mathbb{F}_{q^n}$ with $\alpha+(i-1)\beta$ being $r_i$-primitive and at least one of the elements in the arithmetic progression being $k$-normal over $\mathbb{F}_q$. We obtain asymptotic results for general $k, r_1, \dots, r_m$ and concrete results when $k = r_i = 2$ for $i \in \{1, \dots, m\}$.
\end{abstract}

\vspace{8ex}

\section{Introduction}

For a positive integer $n$ and a prime power $q$, let $\mathbb{F}_{q^n}$ be the finite field with $q^n$ elements. 
We recall that the multiplicative group $\mathbb{F}_{q^n}^*$ is cyclic, and an element $\mathbb{F}_{q^n}$ 
is called \textit{primitive} if its multiplicative order is $q^n-1$. Primitive elements have many applications
in the field of cryptography, see \cite{blum, mel}. Let $r$ be a positive divisor of 
$q^n-1$. An element $\alpha \in \mathbb{F}_{q^n}^*$ is called \textit{$r$-primitive} if its multiplicative order 
is $(q^n-1)/r$. Therefore primitive elements in the usual sense are 1-primitive elements. In \cite{Cohen2021, Cohen2022} the authors found a characteristic function for the $r$-primitive elements. These elements that have high order (small values of $r$), without necessarily being primitive, are of great practical interest because in several
applications they may replace primitive elements.

An element $\alpha \in \mathbb{F}_{q^n}$ is {\em normal} over $\mathbb{F}_q$ if the set 
$B_{\alpha}=\{\alpha, \alpha^q, \ldots, \alpha^{q^{n-1}} \}$ spans
$\mathbb{F}_{q^n}$ as an $\mathbb{F}_q$-vector space. In this case, $B_{\alpha}$ is called a {\em normal basis}. 
Normal bases have many applications in the computational theory due to their efficient implementation in finite field arithmetic \cite{gao}. We will use an equivalence to define the $k$-normal elements (see \cite[Theorem 3.2]{knormal}). An element $\alpha \in \mathbb{F}_{q^n}$ is {\em $k$-normal} over $\mathbb{F}_q$ if $\alpha$ gives rise to a 
basis $\{\alpha,\alpha^{q},\ldots,\alpha^{q^{n-k-1}}\}$ of a $q$-modulus of dimension $n-k$ over $\mathbb{F}_q$. Therefore elements which are normal in the usual sense are $0$-normal. The $k$-normal elements can be used to reduce the multiplication process in finite fields, see \cite{negre}.

If we put these two properties together, we obtain a primitive
normal element. We can study the multiplicative structure of $\mathbb{F}_{q^n}$ while looking at 
$\mathbb{F}_{q^n}$ as a vector space over $\mathbb{F}_q$. \textit{The Primitive Normal Basis Theorem} states that for any extension field $\mathbb{F}_{q^n}$ of $\mathbb{F}_q$, 
there exists a basis composed of primitive normal elements; this result was first proved by Lenstra
and Schoof \cite{lenstra} using a combination of character sums, sieving results and computer
search. There are several criteria in the literature for the existence of $k$-normal elements (see,
for example, \cite{lucas,sozaya,zhang}). In \cite{lucas1}, the authors worked out the case $k = 1$, and established a Primitive 1-Normal Element Theorem. In \cite{AN}, the authors showed conditions for the existence of primitive $2$-normal elements. Generalizing these ideas, in \cite{AN2} and \cite{RST} some conditions for the existence of $r$-primitive, $k$-normal elements in $\mathbb{F}_{q^n}$ over $\mathbb{F}_q$ were obtained. 
In \cite{RSTP}, 
it was established a sufficient condition for the existence of an element $\alpha \in \mathbb F_{q^n}$ such that $\alpha$ and $\alpha^{-1}$ both are simultaneously $r$-primitive and $k$-normal over $\mathbb F_q$. In \cite{AN3}, two of the authors of the present paper studied the pairs $(x,F(x))$, where $F(x) = F_1(x)/F_2(x)$ is a rational function with some conditions on the degrees of $F_1(x)$ and $F_2(x)$, $x$ is $r_1$-primitive and $k_1$-normal, and $F(x)$ is $r_2$-primitive and $k_2$-normal.

In \cite{carlitz2}, Carlitz showed that for every $n$, there exists a number 
$q_0(n)$ such that $\mathbb{F}_q$ contains $n$ consecutive primitive elements for all $q > q_0(n)$. 
For the case $n=2$, in \cite{cohenpairs} Cohen  showed that there exist two consecutive primitive elements in $\mathbb{F}_{q}$, showing that this is valid for $q>7$. In \cite{cohentriples}, the authors worked out the case $n=3$ and showed that $\mathbb{F}_q$ contains three consecutive primitive elements for all odd $q>169$. They also proved that $q_0(n) \leq \exp(2^{5.54n})$ for $n \geq 2$. In \cite{TT} the authors used computational methods (a variant of the `prime divisor tree') in order to prove 
that there are always four consecutive primitive
elements in the finite field $\mathbb{F}_q$ when $q>2401$. 
Considering now an algebraic extension of $\mathbb{F}_q$, in \cite{ASN} the authors showed conditions for the existence 
of arithmetic progression $\{\alpha, \alpha+\beta, \alpha+2\beta, \ldots, \alpha+(m-1)\beta \}$ with $\alpha, \beta \in \mathbb{F}_{q^n}$, such that all these elements are primitive and at least one of them is normal. In this direction, in this paper we are going to generalize the latter for the case when the elements $\alpha+(i-1)\beta$ are $r_i$-primitive for $i \in \{1,\ldots,m\}$ and at least one of them is $k$-normal. 

This paper is organized as follows: In Section 2, we provide a background material that is used along the paper. In Section 3, we present the general condition for the existence of
the aforementioned arithmetic progressions, 
as well as another derived condition using sieve methods.
In Section 4, we provide some numerical examples.

\section{Preliminaries}
In this section, we present some definitions and results required in
this paper.
We refer the reader to \cite{LN} 
for
basic results on finite fields.
For a positive integer $m$, $\phi(m)$ denotes the Euler totient function and $\mu(m)$ denotes
the M\"obius function.
\begin{definition}
\begin{enumerate}
\item[(a)] Let $f(x)\in \mathbb{F}_{q}[x]$. The Euler totient function for polynomials over $\mathbb{F}_q$ is given by
\[
\Phi_q(f)= \left| \left( \dfrac{\mathbb{F}_q[x]}{\langle f \rangle} \right)^{*} \right|,
\]
where $\langle f \rangle$ is the ideal generated by $f(x)$ in $\mathbb{F}_q[x]$.
\item[(b)] If $t$ is a positive integer {\em (}or a monic polynomial over $\mathbb{F}_q${\em )}, $W(t)$ denotes the number of square-free {\em (}monic{\em )} divisors of $t$.
\item[(c)] If $f(x)\in \mathbb{F}_{q}[x]$ is a monic polynomial, the polynomial M\"obius function $\mu_q$ is given by $\mu_q(f)=0$ if $f$ is not square-free and $\mu_q(f)=(-1)^r$ if $f$ is a product of $r$ distinct monic irreducible factors over $\mathbb{F}_q$.
\end{enumerate}
\end{definition}

\subsection{Freeness and characters.}
We present the concept of \textit{freeness}, introduced in Carlitz \cite{carlitz} and Davenport \cite{davenport1},
and refined in Lenstra and Schoof (see \cite{lenstra}). This 
concept is useful in the construction of certain characteristic functions over finite fields.

The additive group $\mathbb{F}_{q^n}$ is an $\mathbb{F}_q[x]$-module where the action
is given by $f \circ \alpha =\displaystyle  \sum_{i=0}^r a_i \alpha^{q^i}$,
for $f=\displaystyle \sum_{i=0}^r a_ix^i\in \mathbb{F}_q[x]$
and $\alpha \in \mathbb{F}_{q^n}$.
An element $\alpha \in \F_{q^n}$ has $\F_q$-order $h\in \mathbb{F}_q[x]$
if $h$ is the lowest degree monic polynomial such that $h \circ \alpha=0$.
The $\mathbb{F}_q$-order of $\alpha$ will be denoted by $\mathrm{Ord} (\alpha)$.
It is known that the $\F_q$-order of an element $\alpha \in \F_{q^n}$
divides $x^n-1$. 
An additive character $\chi$ 
of $\F_{q^n}$ is a
group homomorphism of $\mathbb{F}_{q^n}$ to $\mathbb{C}^*$.
The
group of additive characters $\widehat{\mathbb{F}}_{q^n}$ becomes 
an $\mathbb{F}_q[x]$-module by defining
$f\circ \chi (\alpha)=\chi(f \circ \alpha)$ for $\chi \in \widehat{\mathbb{F}}_{q^n}$,
$\alpha \in \mathbb{F}_{q^n}$ and $f \in \mathbb{F}_q[x]$.
An additive character $\chi$ has $\F_q$-order $h \in \mathbb{F}_q[x]$
if $h$ is the monic polynomial of the smallest degree such
that $h \circ \chi = \chi_0$
is the trivial additive character given by $\chi_0(\alpha) = 1$ for any $\alpha \in \mathbb{F}_{q^n}$.
The $\mathbb{F}_q$-order of $\chi$ will be denoted by $\mathrm{Ord} (\chi)$.

Let $g\in \mathbb{F}_q[x]$ be a  divisor of $x^n-1$. We say that an element
$\alpha \in \mathbb{F}_{q^n}$ is $g$-free
if for every polynomial $h \in \mathbb{F}_q[x]$  such that $h \mid g$ and $h\neq 1$, there is no element
$\beta \in \F_{q^n}$ satisfying $\alpha= h \circ \beta$.
As in the multiplicative case, from \cite[Theorem 13.4.4.]{galois} we have,
for any $\alpha \in \F_{q^n},$ 
\[
\Omega_g(\alpha)=
\Theta(g) \int_{h|g} \chi_h(\alpha)
=\left\{
\begin{array}{ll}
	1, \quad & \text{if } \alpha \text{ is } g\text{-free}, \\
	0,            & \text{otherwise,}
\end{array}
\right.
\]
where 
$\Theta(g)= \frac{\Phi_q(g)}{q^{\deg(g)}}$,
$\displaystyle\int_{h|g} \chi_h$ denotes the sum
$\displaystyle \sum_{h|g} \frac{\mu_q(h)}{\Phi_q(h)} \sum_{(h)} \chi_h$,
$\displaystyle \sum_{h|g}$ runs over all the monic divisors $h\in \mathbb{F}_q[x]$ of $g$,
$\chi_h$ is an additive character of $\F_{q^n}$, and the sum
$\displaystyle \sum_{(h)} \chi_h$ runs over all additive characters of $\mathbb{F}_q$-order $h$.
It is known that there exist $\Phi_q(h)$ of those characters.
It is well known that an element $\alpha \in \mathbb{F}_{q^n}$ is normal if and only if $\alpha$ is $(x^n-1)$-free.

A multiplicative character $\eta$ 
of $\F_{q^n}^*$ is a
group homomorphism of $\F_{q^n}^*$ to $\mathbb{C}^*$.
The group of multiplicative characters $\widehat{\mathbb{F}}_{q^n}^*$ becomes 
a $\mathbb{Z}$-module by defining
$\eta^r(\alpha)=\eta(\alpha^r)$ for $\eta \in \widehat{\mathbb{F}}_{q^n}^*$,
$\alpha \in \mathbb{F}_{q^n}^*$ and $r \in \mathbb{Z}$.
The order of a multiplicative character $\eta$ is the least positive integer $d$ such that
$\eta (\alpha)^d=1$ for any $\alpha \in \F_{q^n}^*$.


There are some works which characterize $r$-primitive elements of $\mathbb{F}_{q^n}$ using characters,
like \cite{Cohen2021} and \cite{Cohen2022}. We will follow \cite{Cohen2022} and, as in that work, for positive integers $a$ and $b$, we set $a_{(b)}=\frac{a}{\gcd(a,b)}$.

\begin{definition}{\cite[Definition 3.1]{Cohen2022}}
	For a divisor $r$ of $q^n-1$
	and a divisor $R$ of $\frac{q^n-1}{r}$, let $\mathcal{C}_r$ be the cyclic multiplicative subgroup of $\mathbb{F}_{q^n}^*$
	of order $\frac{q^n-1}{r}$. We say that an element $\alpha \in  \mathbb{F}_{q^n}^*$ is $(R,r)$-free if
	$\alpha \in \mathcal{C}_r$ and $\alpha$ is $R$-free in $\mathcal{C}_r$, i.e., 
	if $\alpha=\beta^s$ with $\beta \in \mathcal{C}_r$ and $s \mid R$, then $s=1$.
\end{definition}

\begin{remark}
	An element $\alpha \in  \mathbb{F}_{q^n}^*$ is $r$-primitive if and only if $\alpha$ is $(\frac{q^n-1}{r},r)$-free.
\end{remark}

Let $\mathbb{I}_{R,r}$ be the characteristic function of $(R,r)$-free elements of $\mathbb{F}_{q^n}^*$, i.e.,
for $\alpha \in \mathbb{F}_{q^n}^*$,
\[
\mathbb{I}_{R,r}(\alpha) =
\left\{
\begin{array}{ll}
	1, \quad & \text{if } \alpha \text{ is } (R,r)\text{-free}, \\
	0,            & \text{otherwise.}
\end{array}
\right.
\]

Following \cite[Proposition 3.6.]{Cohen2022}, 
for any $\alpha \in \mathbb{F}_{q^n}^*$ we have 
\[\mathbb{I}_{R,r}(\alpha) = 
\frac{\theta(R)}{r} \mathop{\int}_{d_{(r)}|Rr} \eta_d(\alpha),
\]
where $\theta(R)=\frac{\phi(R)}{R}$,
$\displaystyle\mathop{\int}_{d_{(r)}|Rr} \eta_d$ stands for the sum
$\displaystyle \sum_{d|Rr} \frac{\mu(d_{(r)})}{\phi(d_{(r)})} \sum_{(d)} \eta_d$,
$\eta_d$ is a multiplicative character of $\F_{q^n}^*$  and the sum
$\displaystyle \sum_{(d)} \eta_d$ runs over all the multiplicative characters of order $d$.

To bound the sums above, we will use the following particular case of \cite[Lemma 2.5]{Cohen2022}.

\begin{lemma}\label{lemmacohen2022}
	For positive integers $R$, $r$, we have that
	\[
	\sum_{d\mid R} \frac{|\mu(d_{(r)})|}{\phi(d_{(r)})} \cdot \phi(d) = \gcd(R,r)\cdot W(R_{(r)}).
	\]
\end{lemma}

\begin{remark}\label{construct-k-r}
In \cite[Lemma 3.1]{lucas}, Reis provided the following 
method to construct $k$-normal elements. Let $\beta \in \mathbb{F}_{q^n}$ be a normal element and $f\in \mathbb{F}_q[x]$ be a divisor of $x^n-1$ of degree $k$. Then $\alpha = f \circ \beta$ is $k$-normal.
\end{remark}

We also have that $\mathbb{F}_{q^n}^*$ and $\widehat{\mathbb{F}}_{q^n}^*$
are isomorphic as $\mathbb{Z}$-modules, and
$\mathbb{F}_{q^n}$ and $\widehat{\mathbb{F}}_{q^n}$
are isomorphic as $\mathbb{F}_q[x]$-modules (see \cite[Theorem 13.4.1.]{galois}).



We will also need the following definitions.
\begin{definition}\label{char0}
For any $\alpha \in \mathbb{F}_{q^n}$ we define the following character sum:
\[
I_0(\alpha) = \frac{1}{q^n} \sum_{\psi \in \widehat{\mathbb{F}}_{q^n}} \psi(\alpha).
\]
\end{definition}

Note that, by the character orthogonality property, $I_0(\alpha)=1$ if $\alpha=0$, and $I_0(\alpha)=0$ otherwise.

\subsection{Estimates.} 
To finish this section, we present some estimates that are used along the next sections.


\begin{lemma}\label{case-a}
\cite[Theorem 5.6]{Fu} Let $r \in \mathbb{N}$ be a divisor of $q^n-1$,
let $\eta$ be a multiplicative character{\color{red},} and let $\psi$ be a non-trivial  additive character.
Then
\[
\left| \sum_{\alpha \in \mathbb{F}_{q^n}^*} \eta(\alpha) \psi(\alpha^r)\right|
\leq r q^{n/2}.
\]
\end{lemma}

\begin{lemma}\label{case-b}\cite[Lemma 2.5]{AN2} 
Let $f \in \mathbb{F}_q[x]$ be a divisor of $x^n-1$ of degree $k$, and
let $\chi$ and $\psi$ be additive characters. Then
\[
\sum_{\beta \in \mathbb{F}_{q^n}}
\chi(\beta)\psi(f \circ \beta)^{-1} =
\left\{
\begin{array}{ll}
	q^n & \text{if } \chi = f \circ \psi, \\
	0 & \text{if } \chi \neq f \circ \psi .
\end{array}
\right.
\]
Furthermore, for a given additive character $\chi$, 
the set
$\{\psi \in  \widehat{\mathbb{F}}_{q^n} \mid \chi = f \circ \psi\}$
has $q^k$ elements if 
$\mathrm{Ord}(\chi) \mid \frac{x^n-1}{f}$, and it is 
the empty set
if $\mathrm{Ord}(\chi) \nmid \frac{x^n-1}{f}$.
\end{lemma}


%
%

The next result is a combination of \cite[Theorem 5.5]{Fu} and a special case of \cite[Theorem 5.6]{Fu}.

\begin{lemma}\label{cotaparaf}
Let $v(x),u(x) \in \mathbb{F}_{q^n}(x)$ be rational functions. Write $v(x)=\prod_{j=1}^k s_j(x)^{N_v}$, 
where $s_j(x) \in \mathbb{F}_{q^n}[x]$ are irreducible polynomials, pairwise non-associated, and $N_v$ 
are non-zero integers. Let $D_1=\sum_{j=1}^k \deg(s_j)$, $D_2=\max\{\deg(u),0\}$, $D_3$ be the degree of the 
denominator of $u(x)$ and $D_4$ be the sum of degrees of those irreducible polynomials dividing the denominator 
of $u$, but distinct from $s_j(x)$ {\em (}$j\in\{1,\ldots,k\}${\em )}. Let $\chi$ and $\psi$ be, respectively, a multiplicative 
character and a non-trivial additive character of $\mathbb{F}_{q^n}$. Also, denote by $\mathbb F$ the algebraic closure of $\mathbb F_{q^n}$.
\begin{enumerate}
\item[a)] Assume that $v(x)$ is not of the form $r(x)^{ord(\chi)}$ in $\mathbb{F}(x)$. 
Then
\[
\left| \sum_{\substack{\alpha \in \mathbb{F}_{q^n} \\ v(\alpha) \neq 0, v(\alpha) \neq \infty}} 
\chi(v(\alpha)) \right|  \leq (D_1-1)q^{\frac{n}{2}}.
\]
\item[b)] Assume that $u(x)$ is not of the form $r(x)^{q^n}-r(x)$ in $\mathbb{F}(x)$. 
Then
\[
\left| \sum_{\substack{\alpha \in \mathbb{F}_{q^n} \\ v(\alpha) \neq 0, v(\alpha) \neq \infty \\ u(\alpha) \neq \infty}} 
\chi(v(\alpha)) \psi(u(\alpha)) \right|  \leq (D_1+D_2+D_3+D_4-1)q^{\frac{n}{2}}.
\]
\end{enumerate}
\end{lemma}

\section{General results}\label{sectiongen}

Througout this section, $q$ denotes a prime power, $n \ge 2$ is an integer, $m$ is a positive integer smaller than or equal to the characteristic of $\mathbb{F}_{{q^n}}$, $\beta \in \mathbb{F}_{q^n}^*$ is a fixed element, $k$ is a non-negative integer, $r_1,\ldots,r_m$ are positive divisors of $q^n-1$, $R_1,\ldots,R_m$ are divisors of $\frac{q^n-1}{r_1}, \ldots, \frac{q^n-1}{r_m}$, 
and $f,g \in \mathbb{F}_q[x]$ be monic divisors of $x^n-1$ with $\deg f=k$. 

We are interested in finding conditions for the existence of $r$-primitive elements in arithmetic progressions. For this, the following definitions play important roles.

\begin{definition}\label{def-NrfmT}
We denote by $N_v(R_1,\ldots,R_m,g)$ the number of pairs $(\alpha, \gamma) \in \mathbb{F}_{q^n}^* \times \mathbb{F}_{q^n}$
such that $\alpha+(i-1)\beta$ is 
$(R_i,r_i)$-free for all $i \in \{1,\ldots,m\}$ and $\alpha+(v-1)\beta$ is equal to $f \circ \gamma$ with $\gamma$ being $g$-free. 
In particular, if $N_v\left( \frac{q^n-1}{r_1}, \ldots, \frac{q^n-1}{r_m}, x^n-1 \right)>0$, then there exists at least one pair $(\alpha, \gamma) \in \mathbb{F}_{q^n}^* \times \mathbb{F}_{q^n}$ such that, 
for each $i\in \{ 1,\ldots,m\},$ the element
$\alpha+(i-1)\beta$ is $r_i$-primitive and the element $\alpha+(v-1)\beta$ is $k$-normal.
\end{definition}

\begin{definition}\label{def-N}
Assume the notation and conditions of Definition \ref{def-NrfmT}. We denote by $N(R_1,\ldots,R_m,g)$ the number of pairs $(\alpha, \gamma) \in \mathbb{F}_{q^n}^* \times \mathbb{F}_{q^n}$
such that $\alpha+(i-1)\beta$ is 
$(R_i,r_i)$-free for all $i \in \{1,\ldots,m\}$, and at least one element $\alpha+(v-1)\beta$ is equal to $f \circ \gamma$ with $\gamma$ being $g$-free.
In particular, if $N\left( \frac{q^n-1}{r_1}, \ldots, \frac{q^n-1}{r_m}, x^n-1 \right)>0$, then there exists at least one pair $(\alpha, \gamma) \in \mathbb{F}_{q^n}^* \times \mathbb{F}_{q^n}$ such that, for each $i \in \{1,\ldots,m\}$, the element $\alpha+(i-1)\beta$ is $r_i$-primitive, and at least one of them 
is $k$-normal.
\end{definition}

The notations $N(\overline{R},g):=N(R_1,\ldots,R_m,g)$ and $N_v(\overline{R},g):=N_v(R_1,\ldots,R_m,g)$ will be used throughout the text. 
From the previous 
definitions we get
\begin{equation}\label{eq-N}
N(\overline{R},g) \geq \dfrac{1}{m} \sum_{v=1}^m N_v(\overline{R},g).
\end{equation}

We need to find lower bound estimates for the sum above, in order to guarantee the positivity of $N\left( \frac{q^n-1}{r_1}, \ldots, \frac{q^n-1}{r_m}, x^n-1 \right)$. We have the following result.

\begin{theorem}\label{principal}
Let $\widetilde{g} = \gcd(g,\frac{x^n-1}{f})$.
If 
$q^{\frac{n}{2}-k} \geq m W(\widetilde{g})\prod_{i=1}^m r_iW(R_i),$
then $N(\overline{R},g)>0$. In particular, if 
\[q^{\frac{n}{2}-k} \geq m W\left( \frac{x^n-1}{f} \right)\prod_{i=1}^m r_iW \left( \frac{q^n-1}{r_i} \right),\]
then there
exists at least one pair $(\alpha, \gamma) \in \mathbb{F}_{q^n}^* \times \mathbb{F}_{q^n}$ such that, for each $i \in \{1,\ldots,m\}$, the element $\alpha+(i-1)\beta$ is $r_i$-primitive and, for at least one $v \in \{1,\ldots,m\}$, the element $\alpha+(v-1)\beta=f \circ \gamma$ is $k$-normal.
\end{theorem}
\begin{proof}
For $j \in \{1,\ldots,m\}$, from the characteristic functions $\mathbb{I}_{R_i,r_i}$ ($i\in \{1,\ldots,m\}$), $\Omega_{g}$, $I_0$ and Definition \ref{def-NrfmT}, we have
\begin{eqnarray}\label{eq1}
N_v(\overline{R},g) & = &
\sum_{{
			\substack{\alpha \in \mathbb{F}_{q^n} \setminus A \\
				\gamma \in \mathbb{F}_{q^n} }
	}} 
	\left( 
	\prod_{i=1}^m \mathbb{I}_{R_i,r_i}(\alpha+(i-1)\beta)\Omega_g(\gamma)  I_0\left[(\alpha+(v-1)\beta) - f \circ \gamma  \right]  
	\right) \\ \nonumber
	& = & \Theta \mathop{\idotsint}_{
	{d_i}_{(r_i)}|R_ir_i} \mathop{\int}_{h \mid g} \sum_{\psi \in \widehat{\mathbb{F}}_{q^n}} S_v(\eta_{\overline{d}}, \chi_h, \psi),
\end{eqnarray}
where $A=\{-(i-1)\beta \ | \ i \in \{1,\ldots,m\}\}
$, $\Theta= \frac{\theta(R_1)\cdots \theta(R_m) \Theta(g)}{r_1 \cdots r_m}$ and $S_v(\eta_{\overline{d}}, \chi_h,\psi)$ is equal to 
\[
\dfrac{1}{q^n} \sum_{\alpha \in \mathbb{F}_{q^n} \setminus A}  
\eta_{d_1}(\alpha) \cdots \eta_{d_m}(\alpha+(m-1)\beta) \psi(\alpha+(v-1)\beta) \sum_{\gamma \in \mathbb{F}_{q^n}} \chi_h(\gamma) \psi^{-1}(f \circ \gamma).
\]
To find a lower bound for $N(\overline{R},g)$ we will bound 
$\left| S_v(\eta_{\overline{d}}, \chi_h, \psi) \right|$ for $j\in\{1,\ldots,m\}$. First note that, from Lemma \ref{case-b}, 
if $\psi \in \hat{f}^{-1}(\chi_h) := \{\chi \in  \widehat{\mathbb{F}}_{q^n} \mid \chi_h = f \circ \chi\}$, then
\[
\sum_{\gamma \in \mathbb{F}_{q^n}}  \chi_h(\gamma) \psi^{-1}(f \circ \gamma)
= q^{n}.
\]
This sum is $0$ for $\psi \notin \hat{f}^{-1}(\chi_h)$
and $\hat{f}^{-1}(\chi_h)$ is the empty set if  $\mathrm{Ord}(\chi_h) \nmid \frac{x^n-1}{f}$.
Since $\widetilde{g}=\gcd (g, \frac{x^n-1}{f})$, using Inequality \eqref{eq-N} and 
Equation \eqref{eq1}
we get 
\begin{equation}\label{eq2}
N(\overline{R},g) \geq \dfrac{1}{m} \sum_{v=1}^m N_v(\overline{R},g) = \dfrac{\Theta}{m} \mathop{\idotsint}_{
	{d_i}_{(r_i)}|R_ir_i} \mathop{\int}_{h \mid \widetilde{g} } \sum_{\psi \in \hat{f}^{-1}(\chi_{h})} 
	S(\eta_{\overline{d}}, \psi),
\end{equation}
where
\[
S(\eta_{\overline{d}}, \psi)  = \sum_{\alpha \in \mathbb{F}_{q^n} \setminus A}  
\eta_{d_1}(\alpha) \cdots \eta_{d_m}(\alpha+(m-1)\beta) \sum_{v=1}^m \psi(\alpha+(v-1)\beta).
\]

To estimate the previous 
sum we will consider four cases:

\begin{enumerate}

\item[(i)] We first consider the case where $\eta_{d_i}=\eta_1$ is the 
trivial multiplicative character for every $i\in\{1,\ldots,m\}$, and
$\psi=\psi_0$ is 
the trivial additive character. We obtain 
\[
S(\eta_{\overline{1}}, \psi_0) = m (q^n-m).
\]

\item[(ii)] Consider the case where $\eta_{d_i}$ 
is the trivial multiplicative character for every $i\in\{1,\ldots,m\}$, and
$\psi$ is not 
the trivial additive character. Thus 
\[
|S(\eta_{\overline{1}}, \psi)|  = \left| \sum_{v=1}^m \psi(\beta)^{v
-1} \sum_{\alpha \in \mathbb{F}_{q^n} \setminus A} \psi(\alpha) \right| = \left| \sum_{v=1}^m \psi(\beta)^{
v-1} \sum_{\alpha \in A}  																		\psi(\alpha) \right| \leq m^2.
\] 

Before proceeding to treat the cases where at least one multiplicative character is non-
trivial, i.e., not all of $d_1,\ldots,d_m$ have the value 1 (we will denote this case by $\overline{d} \neq \overline{1}$), we 
will rewrite the sum $S(\eta_{\overline{d}}, \psi)$. 
Let $\eta_0$ be a generator of the group of multiplicative characters of $\mathbb F_{q^n}^*$ (see \cite[Corollary 5.9]{LN}). As consequence, for each $i\in \{1,\ldots, m\}$ there exists $c_i \in \{0,1,\ldots , q^n-2\}$ such that $\eta_{d_i}(\alpha)=\eta_0(\alpha^{c_i})$ for all $\alpha \in \mathbb{F}_{q^n}^*$. Observe that $(c_1,\ldots , c_m) \neq \bar{0},$ since $\bar d \neq \bar 1.$ Thus we have

\begin{equation}\label{eq3}
S(\eta_{\overline{d}}, \psi) = \sum_{\alpha \in \mathbb{F}_{q^n}^* \setminus A}  
\eta_{0}(s(\alpha)) \sum_{v=1}^m \psi(\alpha+(v-1)\beta),
\end{equation}
where $s(x)=x^{c_1}(x+\beta)^{c_2} \cdots (x+(m-1)\beta)^{c_m}$. 
\item[(iii)] 
Consider now the case $\overline{d} \neq \overline{1}$ and $\psi=\psi_0$ is the trivial additive character. From Equation \eqref{eq3}, it follows that 
\[
|S(\eta_{\overline{d}}, \psi_0)|= m \left| \sum_{\alpha \in \mathbb{F}_{q^n}^* \setminus A}  
\eta_{d}(s(\alpha)) \right| \leq m (m-1)q^{\frac{n}{2}},
\]
where the previous 
inequality comes from Lemma \ref{cotaparaf}(a), since 
$s(x)$ is not the form $r(x)^{q^n-1}$ in $\mathbb{F}(x)$. 

\item[(iv)] Finally, consider the case $\overline{d} \neq \overline{1}$ and $\psi$ is a non-trivial 
additive character. We use 
Lemma \ref{cotaparaf}(b) since 
$u_v(x)=x-(v-1)\beta$ is not the form $r(x)^{q^n-1}-r(x)$ for any $r(x) \in \mathbb{F}(x)$. From Equation \eqref{eq3} we have
\begin{eqnarray*}
| S(\eta_{\overline{d}}, \psi)| & = \left| \sum_{\alpha \in \mathbb{F}_{q^n}^* \setminus A}  
\eta_{d}(s(\alpha)) \sum_{v=1}^m \psi(\alpha+(v-1)\beta)  \right| \\
& \leq  \left| \sum_{v=1}^m \sum_{\alpha \in \mathbb{F}_{q^n}^* \setminus A}  
\eta_{d}(s(\alpha)) \psi(u_v(x)) \right| \leq m^2 q^{\frac{n}{2}}.
\end{eqnarray*}
\end{enumerate}

We may rewrite the right-hand side of Inequality 
\eqref{eq2} as
\[
\dfrac{\Theta}{m}  (S_1+S_2+S_3+S_4),
\]
where 
\[
\begin{array}{ll}
\displaystyle S_1 = S(\chi_{\bar 1},\eta_0) = m(q^n - m),  
&\displaystyle S_2 = \mathop{\int}_{h \mid \widetilde{g}} \sum_{\substack{\psi \in \hat{f}^{-1}(\chi_{h}) \\ \psi \neq \psi_0}} 
	S(\eta_{\overline{1}}, \psi), \\
\displaystyle S_3 = \mathop{\idotsint}_{ \substack{{d_i}_{(r_i)}|R_ir_i \\ \overline{d} \neq \overline{1}}} S(\eta_{\overline{d}}, \psi_0), \text { and }
&\displaystyle S_4 = \mathop{\idotsint}_{
	{d_i}_{(r_i)}|R_ir_i} \mathop{\int}_{h \mid \widetilde{g} } \sum_{\psi \in \hat{f}^{-1}(\chi_{h})} 
	S(\eta_{\overline{d}}, \psi).
\end{array}
\]

From the previous estimates, we have 
\begin{eqnarray*} 
|S_2| & = & \left| \mathop{\int}_{h \mid \widetilde{g}} \sum_{\substack{\psi \in \hat{f}^{-1}(\chi_{h}) \\ \psi \neq \psi_0}} 
	S(\eta_{\overline{1}}, \psi) \right| 
	\leq  \sum_{h \mid \widetilde{g}} \dfrac{|\mu_q(h)|}{\Phi_q(h)} \sum_{(h)} \sum_{\substack{\psi \in \hat{f}^{-1}(\chi_{h}) \\ \psi \neq \psi_0}} m^2 \\
	& = & m^2(q^k-1) + \sum_{\substack{h \mid \widetilde{g} \\ h \neq 1}} \sum_{(h)} \sum_{\substack{\psi \in \hat{f}^{-1}(\chi_{h}) \\ \psi \neq \psi_0}} m^2 = m^2(q^k-1)+q^km^2(W(\widetilde{g})-1) \\
	& = & m^2 (q^k W(\widetilde{g})-1), \\
|S_3| & = &\left|  \mathop{\idotsint}_{ \substack{{d_i}_{(r_i)}|R_ir_i \\ \overline{d} \neq \overline{1}}} S(\eta_{\overline{d}}, \psi_0) \right| \leq \underbrace{\prod_{i=1}^m \sum_{d_i | R_ir_i} \dfrac{|\mu({d_i}_{(r_i)})|}{\phi({d_i}_{(r_i)})} \sum_{(d_i)}}_{\overline{d} \neq \overline{1}} m(m-1)q^{\frac{n}{2}} \\
	& = & m(m-1)q^{\frac{n}{2}} \left( \prod_{i=1}^m \sum_{d_i | R_ir_i} \dfrac{|\mu({d_i}_{(r_i)})|}{\phi({d_i}_{(r_i)})} \phi(d_i) -\dfrac{\mu(1)}{\phi(1)}  \right) \\
	& = & m(m-1)q^{\frac{n}{2}} \left( \prod_{i=1}^m r_iW(R_i)-1 \right),
\end{eqnarray*}
where the previous 
equality follows from 
Lemma \ref{lemmacohen2022}. Finally, we have

\begin{eqnarray*}
|S_4| & = &\left|  \mathop{\idotsint}_{
	{d_i}_{(r_i)}|R_ir_i} \mathop{\int}_{h \mid \widetilde{g} } \sum_{\psi \in \hat{f}^{-1}(\chi_{h})} 
	S(\eta_{\overline{d}}, \psi) \right| \\
	& \leq & \underbrace{\prod_{i=1}^m \sum_{d_i | R_ir_i} \dfrac{|\mu({d_i}_{(r_i)})|}{\phi({d_i}_{(r_i)})} \sum_{(d_i)}}_{\overline{d} \neq \overline{1}} \sum_{h \mid \widetilde{g}} \dfrac{|\mu_q(h)|}{\Phi_q(h)} \sum_{(h)} \sum_{\substack{\psi \in \hat{f}^{-1}(\chi_{h}) \\ \psi \neq \psi_0}} m^2q^{\frac{n}{2}} \\
\end{eqnarray*}
\begin{eqnarray*}
	& = & m^2q^{\frac{n}{2}} \left( \prod_{i=1}^m \sum_{d_i | R_ir_i} \dfrac{|\mu({d_i}_{(r_i)})|}{\phi({d_i}_{(r_i)})} \phi(d_i) -1 \right)  \left( \sum_{\substack{h \mid \widetilde{g} \\ h \neq 1}} \dfrac{|\mu_q(h)|}{\Phi_q(h)} \Phi_q(h) q^k +q^k-1 \right) \\
	& = & m^2 q^{\frac{n}{2}} \left( \prod_{i=1}^m r_i W(R_i)-1 \right) \left( (W(\widetilde{g})-1)q^k+q^k-1 \right) \\
	& = & m^2 q^{\frac{n}{2}} \left( \prod_{i=1}^m r_i W(R_i)-1 \right) \left( q^kW(\widetilde{g})-1 \right)
\end{eqnarray*}

Using the bounds obtained for each of the sums, we have

\begin{eqnarray*}
\dfrac{m}{\Theta}N(\overline{R},g) & \geq & m (q^n-m) - m^2q^kW(\widetilde{g}) - m^2 - m(m-1)q^{\frac{n}{2}} \left( \prod_{i=1}^m r_iW(R_i)-1 \right)  \\
& - & m^2 q^{\frac{n}{2}} \left( \prod_{i=1}^m r_i W(R_i)-1 \right) \left( q^kW(\widetilde{g})-1 \right) \\
& > & mq^n-m^2 q^{\frac{n}{2}+k}W(\widetilde{g}) \prod_{i=1}^m r_i W(R_i).
\end{eqnarray*}

If $q^{\frac{n}{2}-k} \geq m W(\widetilde{g})\prod_{i=1}^m r_iW(R_i)$, then $N(\overline{R},g)>0$. In particular, 
taking $g=x^n-1$ and $R_i=\frac{q^n-1}{r_i}$ for $i\in\{1,\ldots,m\}$ we obtain the condition of the theorem for the existence of pairs $(\alpha, \gamma) \in \mathbb{F}_{q^n}^* \times \mathbb{F}_{q^n}$ such that 
$\alpha+(i-1)\beta$ is $r_i$-primitive for each $i \in \{1,\ldots,m\}$, 
and at least one element $\alpha+(v-1)\beta=f \circ \gamma$ 
is $k$-normal.

\end{proof}

\subsection{The prime sieve} 
The aim of the section is to relax further the condition
of Theorem \ref{principal}. The sieving technique from the next two results is similar to others which have
appeared in previous works about primitive or normal elements.

\begin{lemma}\label{lema-sieve}
Let $\ell_i$ be a 
divisor of $R_i$ and let $\{p_{i,1},\ldots,p_{i,u(i)}\}$ be the set of all primes which
divide $R_i$ but do not divide $\ell_i$, for all $i\in\{1,\ldots,m\}$.
Also, 
let $\{h_1,\ldots,h_s\}$ be the set of all monic irreducible polynomials which divide $x^n-1$ but do not divide $g$. Then, for $v\in\{1,\ldots,m\}$, we get
\begin{eqnarray}\label{desig-sieve}
N_v(\overline{R},x^n-1) & \geq & \sum_{j=1}^{u(1)} N_v(\ell_1p_{1,j},\ldots, \ell_m,g) + 
							 \sum_{j=1}^{u(2)} N_v(\ell_1,\ell_2p_{2,j},\ldots, \ell_m,g)  \\ \nonumber
					  &	+ \cdots + & \sum_{j=1}^{u(m)} N_v(\ell_1, \ldots, \ell_mp_{m,j},g) + \sum_{j=1}^s N_v(\ell_1,\ldots,											\ell_m,gh_j) \\ \nonumber
					  &	 - & (u(1)+\cdots +u(m)+s-1)N_v(\ell_1,\ldots,\ell_m,g).
\end{eqnarray}
\end{lemma}

\begin{proof}
The left-hand side of \eqref{desig-sieve} counts 
the numbers of pairs $(\alpha, \gamma) \in \mathbb{F}_{q^n}^* \times \mathbb{F}_{q^n}$ 
such that $\alpha+(i-1)\beta$ is 
$(R_i,r_i)$-free for all $i \in \{1,\ldots,m\}$ and $\alpha+(v-1)\beta=f \circ \gamma$, with $\gamma$ 
being $k$-normal. Observe that for such pairs $(\alpha, \gamma)$ we also have that
$\alpha$ is  $(\ell_i,r_i)$-free and $(\ell_ip_{i,j},r_1)$-free for $j\in\{1,\ldots,u(i)\}$ and $i\in\{1,\ldots,m\}$, 
$\gamma$ is $g$-free and $(gh_j)$-free for $j\in\{1,\ldots,s\}.$ Therefore 
$(\alpha, \gamma)$ is counted $u(1)+\ldots +u(m)+w - (u(1)+\ldots +u(m)+w-1)=1$ times 
on the right-hand side of \eqref{desig-sieve}.

For any other 
pair $(\alpha, \gamma) \in \mathbb{F}_{q^n}^* \times \mathbb{F}_{q^n}$,
we have that either $\alpha$ is not $(\ell_ip_{i,j},r_i)$-free for some
$i \in \{ 1,\ldots , m\}$ and some $j \in \{1,\ldots,u(i)\}$ or $\gamma$ is not $gh_j$-free for some 
$j \in \{1,\ldots,w\}$ 
or $\alpha \neq f \circ \gamma$, therefore the right-hand side of \eqref{desig-sieve} is at most zero.
\end{proof}

\begin{proposition}\label{prop1-crivo}  
Assume the notation and conditions of Lemma \ref{lema-sieve}
and assume also that 
the polynomials of the set $\{h_1,\ldots,h_s\}$ divide $\frac{x^n-1}{f}$.
Let $\widetilde{g}=\gcd (g, \frac{x^n-1}{f})$,
$\delta=1-\sum_{i=1}^m \sum_{j=1}^{u(i)} \frac{1}{p_{i,j}}-\sum_{j=1}^s \frac{1}{q^{\deg(h_j)}}>0$, and $\Delta=2+\frac{u(1)+\ldots+u(m)+s-1}{\delta}$. 
If 
\begin{equation}\label{eq-prop1-crivo}
q^{\frac{n}{2}-k} \geq m \Delta W(\widetilde{g}) \prod_{i=1}^m r_iW(\ell_i),
\end{equation}
then $N(\overline{R},g)>0$.
\end{proposition}

\begin{proof}
We can rewrite Inequality \eqref{desig-sieve} in the form
\begin{eqnarray}\label{eq1-crivo}
N_v(\overline{R},x^n-1) & \geq & \sum_{j=1}^{u(1)} \left[ N_v(\ell_1p_{1,j},\ldots,\ell_m,g)- \theta(p_{1,j}) 												N_v(\overline{\ell},g) \right] + \\ \nonumber
				   & + \cdots + & \sum_{j=1}^{u(m)} \left[ N_v(\ell_1,\ldots,\ell_m p_{m,j},g)- \theta(p_{m,j}) 											N_v(\overline{\ell},g) \right] + \\ \nonumber
				   & + & \sum_{j=1}^s \left[ N_v(\ell_1,\ldots,\ell_m,gh_j)- \Theta(h_j) N_v(\overline{\ell},g) \right] + 						\delta N_v(\overline{\ell},g),
\end{eqnarray}
where $N_v(\overline{\ell},g)=N_v(\ell_1,\ldots,\ell_m,g)$. For all $i \in \{1,\ldots,m\}$, let $j \in \{1,\ldots,u(i)\}$. From 
Definition \ref{def-NrfmT},
taking into account that $\theta$ is a multiplicative function, we get 
\begin{eqnarray*}
N&=&\dfrac{\theta(\ell_i)\theta(p_{i,j}) \Theta(g) \prod \limits_{\substack{w=1 \\ w \neq i}}^{m}\theta(\ell_w)}{r_1 r_2 \cdots r_m} \mathop{\idotsint}_{\substack{{d_i}_{(r_i)}|\ell_ip_{i,j}r_i \\
	{d_w}_{(r_w)}|\ell_wr_w, \ w \neq i}} \mathop{\int}_{h \mid g} \sum_{\psi \in \widehat{\mathbb{F}}_{q^n}} S_v(\eta_{\overline{d}}, \chi_h, \psi)\\
&= &\theta(p_{i,j}) N_v(\overline{\ell},g) + \theta(p_{i,j}) \Theta_{\ell}
\mathop{\idotsint}_{\substack{{d_i}_{(r_i)}|\ell_ip_{i,j}r_i \\ p_{i,j} | {d_i}_{(r_i)} \\ {d_w}_{(r_w)}|\ell_wr_w, \ w \neq i}  } \mathop{\int}_{h \mid g} \sum_{\psi \in \widehat{\mathbb{F}}_{q^n}} S_v(\eta_{\overline{d}}, \chi_h, \psi),
\end{eqnarray*}
where $N=N_v(\ell_1,\ldots,\ell_ip_{i,j},\ldots,\ell_m,g)$ and $\Theta_{\ell}=\frac{\Theta(g)\theta(\ell_1) \cdots \theta(\ell_m)}{r_1\cdots r_m}$. From the proof of Theorem \ref{principal} (see the sum $S_4$), 
we get
\[
\Big|
\mathop{\idotsint}_{\substack{{d_i}_{(r_i)}|\ell_ip_{i,j}r_i \\ p_{i,j} | {d_i}_{(r_i)} \\ {d_w}_{(r_w)}|\ell_wr_w, \ w \neq i}  } \mathop{\int}_{h \mid g} \sum_{\psi \in \widehat{\mathbb{F}}_{q^n}} S_v(\eta_{\overline{d}}, \chi_h, \psi) 
\Big|
\leq 
mq^{\frac{n}{2}+k} W(\widetilde{g})\prod_{i=1}^k r_iW(\ell_i).
\]
Then, for $i\in\{1,\ldots,m\}$ and $j\in\{1,\ldots,u(i)\}$, we have
\[
|N_v(\ell_1,\ldots,\ell_ip_{i,j},\ldots,\ell_m,g) -  \theta(p_{i,j})N_v(\overline{\ell},g)|\leq  \theta(p_{i,j}) \Theta_{\ell}
mq^{\frac{n}{2}+k} W(\widetilde{g})\prod_{i=1}^k r_iW(\ell_i).
\]
Similarly, for $j\in\{1,\ldots,s\},$ we get
\[
|N_v(\ell_1,\ldots,\ell_m,gh_j) - \Theta(h_j) N_v(\overline{\ell},g)| \leq \Theta(h_j) \Theta_{\ell} mq^{\frac{n}{2}+k} W(\widetilde{g})\prod_{i=1}^k r_iW(\ell_i).
\]
Therefore, combining the previous inequalities, we obtain
\[
N_v(\overline{R},x^n-1) \geq \delta N_v(\overline{\ell},g) - \Theta_{\ell} mq^{\frac{n}{2}+k} W(\widetilde{g})\prod_{i=1}^k r_iW(\ell_i) \left( \sum_{i=1}^m \sum_{j=1}^{u(i)} \theta(p_{i,j}) + \sum_{j=1}^{s} \Theta(h_j) \right).
\]
From the proof of Theorem \ref{principal}, we have
\begin{eqnarray*}
N_v(\overline{R},x^n-1) & \geq & \delta \Theta_{\ell} \left( q^n-mq^{\frac{n}{2}+k}W(\widetilde{g})\prod_{i=1}^k r_iW(\ell_i)\right) \\
	& &- \Theta_{\ell} mq^{\frac{n}{2}+k} W(\widetilde{g})\prod_{i=1}^k r_iW(\ell_i) \times 
	\left( \sum_{i=1}^m \sum_{j=1}^{u(i)} \theta(p_{i,j}) + \sum_{j=1}^{s} \Theta(h_j) \right) \\
	& = & \delta \Theta_{\ell} \left( q^n - m \Delta q^{\frac{n}{2}+k}W(\widetilde{g})\prod_{i=1}^k r_iW(\ell_i)  \right),
\end{eqnarray*}
hence from Inequality \eqref{eq-N} we obtain the desired result.
\end{proof}

To apply Theorem \ref{principal}, in order to obtain asymptotic results, we need (among
other results) an upper bound of $W(u)$.

\begin{proposition}[{\cite[Proposition 4.1]{ASN}}]\label{boundW}
	Let $e$ be a positive integer,
	$p_1, \ldots , p_e$ be the list of the first $e$ prime numbers, and $\mathcal{P}_e = p_1 \cdot \ldots \cdot p_e$
	be its product.
	For every positive integer $u \geq \mathcal{P}_e$, we have $W(u) \leq u^{\frac{1}{N}}$, where $N$ is a real number
	satisfying $\frac{1}{N} \geq \frac{e \log 2}{\log \mathcal{P}_e}$.
\end{proposition}

\begin{proposition}\label{cotainicial}
Let $N$ be a positive integer such that $\frac{1}{2}-\frac{1}{N} - \log_q 2 >0$ and
let $e$  be a positive integer such that
$\frac{1}{Nm} \ge \frac{e \log 2}{\log \mathcal{P}_e}$
and $\frac{q^n-1}{r_i} \ge \mathcal{P}_e$ for all $i \in \{ 1,\ldots, m\}$.
Suppose that
\begin{equation}\label{condicaobizarra}
q^{n
\left(
\frac{1}{2} - \frac{1}{N} - \log_q 2
\right) - k
} 
\ge \frac{m (r_1 \cdots r_m)^{1 - \frac{1}{Nm}}}{2^k}.
\end{equation}
Then for all $\beta \in \mathbb{F}_{q^n}^*,$
there exists an element $\alpha\in \mathbb{F}_{q^n}^*$
for which, for every $i\in \{ 1,\ldots,m\}$,
the element
$\alpha+(i-1)\beta$ is $r_i$-primitive
and, for some $v \in \{ 1,\ldots,m\}$, the element $\alpha+(v-1)\beta$ is $k$-normal
over $\mathbb{F}_q.$
\end{proposition}
\begin{proof}
From Proposition \ref{boundW}, we have
$W(\frac{q^n-1}{r_i})\le \left( \frac{q^n}{r_i}\right)^{\frac{1}{Nm}}$
for all $i \in \{ 1,\ldots,m\}$.
%
%
%
In this case, we have
\[
m W\left( \frac{x^n-1}{f} \right)\prod_{i=1}^m r_iW \left( \frac{q^n-1}{r_i} \right)
\le m \cdot 2^{n-k} \cdot \left( r_1 \cdots r_m \right)^{1-\frac{1}{Nm}}
\cdot q^{\frac{n}{N}}.
\]
Observe that Inequality \eqref{condicaobizarra} is equivalent to
$q^{\frac{n}{2}-k} \ge m \cdot (r_1 \cdots r_m)^{1 - \frac{1}{Nm}} \cdot 2^{n-k} \cdot q^{\frac{n}{N}}.$
Combining both inequalities, we have
\[
q^{\frac{n}{2}-k} \ge m W\left( \frac{x^n-1}{f} \right)\prod_{i=1}^m r_iW \left( \frac{q^n-1}{r_i} \right)
\]
and we get the desired result from Theorem \ref{principal}.
\end{proof}

\section{Arithmetic progressions with $r_i=2$ and $k=2$}

In this section, we are going to deal with the particular case
where $m =3$, $r_i=2$ for $i \in \{ 1,\ldots, m\}$ and $k=2$. All the procedures and numerical calculations are done using SageMath \cite{SAGE}.

Let $A$ be the set of pairs $(q,n)$ such that there exist $2$-primitive elements $\alpha, \alpha + \beta, \alpha+2\beta \in \mathbb{F}_{q^n}$ with at least one of them being $2$-normal. From \cite[Lemma 3.1]{AN}, if
$(q,n)\in A$, then $q$ is odd and $\gcd(q^3 -q,n)\ne 1$.
From now on, we will assume these two conditions hold.

In order to study the existence of such elements, we 
need to bound the function $W(x^n - 1).$

\begin{lemma}[{\cite[Lemma 4.3]{ASN}}]\label{cota-xn-1}
Let $n\ge 5$ be a positive integer and let $f$ be a quadratic factor of $x^n-1$
over $\mathbb{F}_q$. The number of monic
irreducible factors of
$\dfrac{x^n-1}{f}$ over $\mathbb{F}_q$ is at most 
$n-2$ or
$\frac{n}{a}+b-1,$ where
the pair $(a,b)$ can be chosen among the following pairs:
\[
\left(2,\frac{q-1}2\right), \quad \left(3, \frac{q^2+3q-4}{6}\right),\quad  
\left(4,\frac{q^3+3q^2+5q-9}{12}\right).
\]
\end{lemma}

\begin{proposition}\label{firstbound}
If $n \geq13$ and $q^n \geq 4.12 \cdot 10^{718}$, then $(q,n) \in A.$
\end{proposition}
\begin{proof}
Note that $\frac{1}{9} \ge \frac{e \log 2}{\log \mathcal{P}_e}$
is satisfied for $e\ge 265$. From Proposition \ref{cotainicial} with $N=3$ and
$e=265$, we have that 
if
$q^n \ge 4.12\cdot 10^{718} > 2 \cdot \mathcal{P}_{e}$ and
\begin{equation}\label{bizarra-exemplo}
q^{n
	\left(
	\frac{1}{6} - \log_q 2
	\right) - 2
} 
\ge 3 \cdot 2^{\frac{2}{3}},
\end{equation}
then $(q,n)\in A$.
Applying logarithms and
using the fact $q^n \geq 4.12\cdot 10^{718}$, the inequality
\begin{equation}\label{bizarro2}
\frac{1}{6}- \frac{2}{n}- \frac{\log 2}{\log q} \geq \frac{\log (3 \cdot 2^{\frac{2}{3}})}{\log (4.12\cdot 10^{718})}
\end{equation}
implies Inequality \eqref{bizarra-exemplo}.
We define $y(n,q):=\frac{1}{6}- \frac{2}{n}- \frac{\log 2}{\log q}$ and we study
the inequality
$y(n,q) \geq 0.00095 > \frac{\log (3 \cdot 2^{\frac{2}{3}})}{\log (4.12\cdot 10^{718})}$. 
Note that $(2 \mathcal{P}_e)^{\frac{1}{n}} \le 307$ if and only if
$n \ge \frac{\log (2\mathcal{P}_e)}{\log 307}$.
Since $288.94 > \frac{\log (2\mathcal{P}_e)}{\log 307}$, if 
$n \geq 289$ and $q \geq 307$, 
then $q \ge (2 \mathcal{P}_e)^{\frac{1}{n}}$ and
$y(q,n) \geq \frac{1}{6}- \frac{2}{289}- \frac{\log 2}{\log 307} \ge 00095$, which implies that $(q,n) \in A$. 

Suppose $q \geq 79$, from the condition $q^n \geq 2\mathcal{P}_e$ we have $\log q \geq \max \{ \frac{\log 2\mathcal{P}_e}{n}, \log 79 \}$ and 
\[
y(n,q) \geq z(n), \ \ \ \text{where} \ \ \ 
z(n)= \left\{ \begin{array}{lll}
             \dfrac{1}{6}-\dfrac{2}{n}-\dfrac{n\log 2}{\log 2\mathcal{P}_e} & \text{if} & n< \dfrac{\log 2\mathcal{P}_e}{\log 79} \\
              \dfrac{1}{6}-\dfrac{2}{n}-\dfrac{\log 2}{\log 79}& \text{if} & n \geq \dfrac{\log 2\mathcal{P}_e}{\log 79}
             \end{array}
   \right. .
\]
If $ n \geq \frac{\log 2\mathcal{P}_e}{\log 79}$, we have $y(n,q) \geq z(n) \geq 0.00095$, since $z\left(\frac{\log 2\mathcal{P}_e}{\log 79} \right) \geq 0.00095$ and $z(n)$ is an increasing function. Now the inequality $y(n,q) \geq z(n) \geq 0.00095$ is valid for
$13 \leq n\le 377 < \dfrac{\log 2\mathcal{P}_e}{\log 79}$.

Using Lemma \ref{cota-xn-1}, we can use the bound $W(x^n-1) \leq 2^{\frac{n}{a}+b}$
to get $W(\frac{x^n-1}{f}) \leq 2^{\frac{n}{a}+b-1}.$ This implies
\[
3\cdot  W\left( \frac{x^n-1}{f} \right)\cdot 2^3  \cdot W \left( \frac{q^n-1}{2} \right)^3
\le 3 \cdot 2^{\frac{n}{a}+b-1} \cdot \left( 2^3 \right)^{1-\frac{1}{9}}
\cdot q^{\frac{n}{3}}.
\]
Thus, from Theorem \ref{principal}, if
\[
q^{n
	\left(
	\frac{1}{6} - \frac{1}{a}\log_q 2
	\right) - 2
} 
\ge 3 \cdot 2^{b+\frac{5}{3}},
\]
then $(q,n)\in A$. This allows us to rewrite Inequality \eqref{bizarro2} as
\begin{equation}\label{cond-mod1}
\frac{1}{6}- \frac{2}{n}- \frac{\log 2}{a \log q} \geq \frac{\log \left(3 \cdot 2^{b+ \frac{5}{3}} \right)}{\log 2 \mathcal{P}_e}.
\end{equation}
We will use this 
condition to decrease the bound of $q$. Suppose $q<79$, from Lemma \ref{cota-xn-1} with $a=2$ and $b=q-1$, Inequality \eqref{cond-mod1} holds for $q>7$. For $q \in \{5,7\}$, Inequality \eqref{cond-mod1} holds
with $a=3$ and $b=\frac{q^2+3q-4}{6}$. Finally, for $q=3$, Inequality \eqref{cond-mod1} holds with $a=4$ and $b=\frac{q^3+3q^2+5q-9}{12}$. This completes the proof.
\end{proof}

\begin{proposition}\label{bound2}
If $n \geq13$ and $q \geq 79$, then $(q,n) \in A.$
\end{proposition}
\begin{proof}
We use the sieve method given by Proposition \ref{prop1-crivo} in order to decrease the bound for $q^n$. From the 
previous proposition, we may consider $q^n< q_{\max}:=4.12 \cdot 10^{718}$. 
Fix a prime number $p_0$, let $g=x^n-1$ and, for $i \in \{1,2,3\}$, let $\ell_i$ be the product of prime numbers which divide $\frac{q^n-1}{2}$ less than $p_0$. Denote by $t$ the number of prime numbers less than $p_0$ which divide $q^n-1$, so that $t < \pi(p_0)$ and $W(\ell_i)= 2^t$. 

Let also $\{ p_1, \ldots , p_u \}$ be the set of primes
which divide $q^n-1$ and are greater than or equal to $p_0$ and
denote by $\mathcal{P}(u,p_0)$, $\mathcal{S}(u,p_0)$ the product and the sum of the inverses, respectively, of the first $u$ prime numbers greater than or equal to $p_0$. Therefore,
\[\mathcal{P}(u,p_0) \le p_1 \cdots  p_u  \ \ \ \text{and} \ \ \ 
\sum_{i=1}^u \frac{1}{p_i} \le \mathcal{S}(u,p_0).
\]
Let $u(t)=\max \{ u \mid 2 \cdot \mathcal{P}_t \cdot \mathcal{P}(u,p_0) \leq q_{\max} \}$,
where $\mathcal{P}_t$ is the product of the first $t$ prime numbers. 
From Proposition \ref{prop1-crivo}, we get
$\delta \geq 1-3\mathcal{S}(u(t),p_0)$ and $\Delta \leq 2 + \frac{3u(t)-1}{1-3\mathcal{S}(u(t),p_0)}=: \Delta(t)$.
We need to choose $p_0$ such that $1-3\mathcal{S}(u(t),p_0)>0$ in order to use
Proposition \ref{prop1-crivo}. 
Observe that if
$
q^{\frac{n}{2}-2} \geq 3 \ \Delta(t) \ 2^{3t+n+1}
$
for some $2 \le t \le \pi(p_0-1)$, then Inequality \eqref{eq-prop1-crivo} holds. Using the fact that $n \geq 13$ and considering $q \geq 79$, we can rewrite the previous condition as
\begin{equation}\label{eq1-bound2}
\left( q^n \right)^{\frac{1}{2}-\frac{2}{13}-\log_{79} 2} \geq 3 \ \Delta(t) \ 2^{3(t+1)}.
\end{equation}
Using $p_0=223$, we have that Inequality \eqref{eq1-bound2} holds for $q^n \geq 4.572 \cdot 10^{252}$. Repeating the process with 
$p_0=107,73,67,61, 59,59$ sequentially, we obtain the bound $q^n \geq 1.101 \cdot 10^{97}$. 
Note that if $n \geq 52$, then $q^n > 1.101 \cdot 10^{97}$, because $q \geq 79$. 
Hence it is only required to consider the cases $13 \leq n \leq 51$.

Now we use the Sagemath procedure \textbf{\ref{specialsieve}} (see Appendix \ref{apendice}). The value of $\Delta$ given in \textbf{\ref{specialsieve}} 
line \ref{Delta} is greater than or equal to
$\Delta$ from Proposition \ref{prop1-crivo}, since the pair $(S,u_0)$
given by Procedure \textbf{\ref{sumfactors}} 
satisfies $u(i) \le u_0 $ for $i \in \{1,2,3 \}$ and $\delta \ge 1-3S$ for
$(u,\delta)$ in Proposition \ref{prop1-crivo}. If \textbf{\ref{specialsieve}($q,n,p_0$)} 
returns True, then Inequality \eqref{eq-prop1-crivo} holds. Using this procedure we obtain that $(q,n) \in A$ for $13 \leq n \leq 51$ and $q \ge 79$.
\end{proof}

\section*{Acknowledgements}

The authors would like to thank the referee for the careful reading and helpful comments that improved the presentation of the paper.

The authors were partially supported by FAPEMIG grant RED-00133-21, the second and fourth authors were partially supported by FAPEMIG grant APQ-02546-21, the third author was partially supported by FAPEMIG grant APQ-00470-22, and the fourth author was also partially supported by NAWI Graz Postdoctoral Program.

\newpage

\appendix

\section{Procedures in SageMath}\label{apendice}

{\scriptsize
	\begin{procedure}
		\KwIn{A prime power $q$ and a positive integer $n$}
		\KwOut{A non negative integer or $\emptyset$}
		$w \gets$ number of monic irreducible factors of $x^n-1$ in $\mathbb{F}_q[x]$\\
		\uIf{$\gcd(q,n)>1$}
			{$Length \gets w$}
		\uElseIf{$\gcd(q-1,n)>1$}
			{$Length \gets w-2$}
		\uElseIf{$\gcd(q+1,n)>1$}
			{$Length \gets w-1$}
		\Else{$Length \gets \emptyset$}
        \Return $Length$
		\caption{NumberPolFactors($q,n$)}\label{valw1w2}
	\end{procedure}
}

{\scriptsize
\begin{procedure}
\KwIn{A positive integer $T$ and a prime number $p_0$}
\KwOut{A pair $(S,u_0)$ where $S$ is a positive real number and $u_0$ is a positive integer}
$(p,S,u_0) \gets (p_0,0,0)$\\
\While{$T\ge p$ and $p<1000$}
	{
	\uIf{$p$ divides $T$}
		{
		$T \gets \frac{T}{p}$\\
		$S\gets S+ \frac{1}{p}$\\
		$u_0 \gets u_0+1$\\
		\While{$p$ divides $T$}
			{
			$T \gets \frac{T}{p}$\\
			}
		$p \gets$ next prime after $p$
		}
	\Else{
		$p \gets$ next prime after $p$
		}
	}
\While{$p<T$}
	{
	$T \gets \frac{T}{p}$\\
	$S\gets S+ \frac{1}{p}$\\
	$u_0 \gets u_0+1$\\
	$p \gets $ next prime after $p$
	}
\Return $(S,u_0)$
\caption{SumFactors($T,p_0$)}\label{sumfactors}
	\end{procedure}
}

{\scriptsize
\begin{procedure}
\KwIn{A prime power $q$, a positive integer $n$ and a prime number $p_0$}
\KwOut{True or False}
$w_1 \gets$ \ref{valw1w2}($q,n$)\\
\uIf{$w_1 \neq \emptyset$}
	{
	$T \gets \frac{q^n-1}{2}$\\
	$p \gets 2$\\
	$\ell \gets 1$\\
	\While{$p<p_0$}
		{
		\While{$p$ divides $T$}
			{
			$T \gets \frac{T}{p}$\\
			$\ell \gets \ell \cdot p$
			}
		$p \gets $ next prime after $p$
		}
	$w_2 \gets$ number of prime divisors of $\ell$\\ 
	$(S,u_0) \gets $ \ref{sumfactors}$(T,p_0)$\\
	$\delta \gets 1 - 3 S$\\
	\uIf{$\delta>0$}
		{
		$\Delta \gets 2+\frac{3\cdot u_0-1}{\delta}$  \nllabel{Delta}\\
		$res\gets \left[ q^{n/2-2}\ge 3 \cdot \Delta \cdot 2^{w_1+3w_2+3}\right]$
		}
	\Else{$res \gets$ False}
	}
\Else{$res \gets$ False}
\Return $res$
\caption{SpecialSieve($q,n,p_0$)}\label{specialsieve}
\end{procedure}
}


\begin{thebibliography}{99}

\bibitem{AN2}
J.J.R. Aguirre, C. Carvalho and V.G.L. Neumann, {\em About r-primitive and k-normal elements in finite fields}, Des. Codes Cryptogr. (2022), 1--12.

\bibitem{AN}
J.J.R. Aguirre and V.G.L. Neumann,
{\em Existence of primitive $2$-normal elements in finite fields},
Finite Fields Appl. 73 (2021), 101864.
170--183.

\bibitem{AN3}
J.J.R. Aguirre and V.G.L. Neumann,
{\em Pairs of r-primitive and k-normal elements in finite fields},
Bull. Braz. Math. Soc., New Series 54: 24 (2023).


\bibitem{blum} 
M. Blum and S. Micali, {\em How to generate cryptographically strong sequences of pseudorandom bits}, SIAM J. Computing 13 (1984), 850--864.

\bibitem{carlitz} L. Carlitz, {\em Primitive roots in a finite field}, Trans. American Math. Soc. 73 (1952), 373--382.

\bibitem{carlitz2} L. Carlitz, {\em Sets of primitive roots}, Compos. Math. 13 (1956), 65--70. 



\bibitem{cohenpairs} S.D. Cohen, {\em Consecutive primitive roots in a finite field}, 
Proc. Amer. Math. Soc. 93 (1985), 189--197.

\bibitem{Cohen2021} S.D. Cohen and G. Kapetanakis, {\em Finite field extensions with the line or translate property for $r$-primitive elements},  J. Aust. Math. Soc. 111 (2021), no. 3, 313--319.

\bibitem{Cohen2022} S.D. Cohen, G. Kapetanakis and L. Reis,
{\em The existence of $\mathbb{F}_q$-primitive points on curves using freeness},
Comptes Rendus. Math\'ematique 360 (2022), 641--652

\bibitem{cohentriples} S.D. Cohen, T. Oliveira e Silva and T. Trudgian, 
{\em On consecutive primitive elements in a finite field}, Bull. London Math. Soc. 47 (2015), 418--426.

%
%
%
%




%

\bibitem{davenport1} H. Davenport, {\em Bases for finite fields}, J. London Math. Soc. 43 (1968), 21--39.


\bibitem{Fu} L. Fu and D.Q. Wan, {\em A class of incomplete character sums}, Quart. J. Math. 65 (2014), 1195--1211.

\bibitem{gao} S. Gao, {\em Elements of provable high orders in finite fields}, Proc. American Math. Soc. 127 (1999), 1615--1623.

\bibitem{galois} R. Hachenberger and D. Jungnickel, {\em Topics in Galois Fields}, Springer, 2020.







%

\bibitem{knormal} S. Huczynska, G.L. Mullen, D. Panario and D. Thomson,
{\em Existence and properties of $k$-normal elements over finite fields}, Finite Fields Appl. 24 (2013),
170--183.

\bibitem{TT} T. Jarso and T. Trudgian, {\em Four consecutive primitive elements in a finite field}, Math. Comp. 91 (2022), 1521--1532.



\bibitem{ASN} A. Lemos, V.G.L. Neumann and S. Ribas, {\em On arithmetic progressions in finite fields}, 
Des. Codes Cryptogr. 91 (2023), 2323--2346 .

\bibitem{lenstra} 
H.W. Lenstra and R. Schoof, {\em Primitive normal bases for finite fields}, Math. Comp. 48 (1987), 217--231.


\bibitem{LN} R. Lidl and H. Niederreiter, {\em Finite Fields}, Cambridge University Press (1997).

\bibitem{mel} G. Meletiou and G. Mullen, {\em A note on discrete logarithms in finite fields}, Appl. Algebra Engrg. Comm. Comput. 3(1) (1992), 75--78.

\bibitem{negre}
C. Negre, {\em Finite field arithmetic using quasi-normal bases}, Finite Fields Appl. 13 (2007), 635--647.



\bibitem{RST}
M. Rani, A.K. Sharma and S.K. Tiwari,
{\em On $r$-primitive $k$-normal elements over finite fields},
Finite Fields Appl. 82 (2022), 102053.

\bibitem{RSTP}
M. Rani, A.K. Sharma and S.K. Tiwari, A. Panigrahi,
{\em Inverses of $r$-primitive $k$-normal elements over finite fields},
preprint available at: \url{https://arxiv.org/pdf/2201.11334.pdf} (2022).

\bibitem{lucas} L. Reis,
{\em Existence results on $k$-normal elements over finite fields},
Rev. Mat. Iberoam. 35(3) (2019),
805--822.

\bibitem{lucas1}
L. Reis and D. Thompson, {\em Existence of primitive $1$-normal elements in finite fields}, Finite Fields Appl. 51 (2018), 238--269.

\bibitem{SAGE} The Sage Developers,
SageMath, the Sage Mathematics Software System (Version 8.1),
\url{https://www.sagemath.org}, 2020.



%
%


\bibitem{sozaya} J.A. Sozaya-Chan and H. Tapia-Recillas, {\em On $k$-normal elements over finite fields}, Finite Fields Appl. 52 (2018), 94--107.

\bibitem{zhang} A. Zhang and K. Feng, 
{\em A New Criterion on $k$-Normal Elements over Finite Fields}, Chinese Annals Math. Series B 41 (2020), 665--678.

\end{thebibliography}
\end{document}